\documentclass[a4paper,english]{lipics-v2021}
\hideLIPIcs

\usepackage{amsthm}
\usepackage{amsmath}
\usepackage{amssymb}
\usepackage{amsfonts}
\usepackage{epsfig}
\usepackage{graphics}
\usepackage{graphicx}
\usepackage{booktabs}
\usepackage{pdfpages}
\usepackage{tikz}
\usepackage{MnSymbol}

\nolinenumbers

\usepackage{complexity}
\usepackage{url}

\usepackage{latexsym}
\usepackage{psfrag}
\usepackage{enumerate}
\usepackage{here}
\usepackage{esvect}
\usepackage[update,prepend]{epstopdf}
\usepackage{anyfontsize}
\usepackage{url}
\usepackage{hyperref}

\usepackage{microtype}

\newcommand{\ie}{{i.e.}}

\newcommand{\area}{{\rm area}}

\newcommand{\NN}{\mathbb{N}} 

\def\E{{\rm E}}

\providecommand{\intd}[0]%
{\;\mbox{d}}

\usepackage[noline,linesnumbered]{algorithm2e}
\SetArgSty{textrm}
\let\oldnl\nl
\newcommand{\nonl}{\renewcommand{\nl}{\let\nl\oldnl}}
\makeatletter
\def\TitleOfAlgo{\@ifnextchar({\@TitleOfAlgoAndComment}{\@TitleOfAlgoNoComment}}
\def\@TitleOfAlgoAndComment(#1)#2{\nonl\hspace*{-1.5em}#2 #1\;}
\def\@TitleOfAlgoNoComment#1{\nonl\hspace*{-1.5em}#1\;}
\makeatother
\DontPrintSemicolon

\newcommand{\later}[1]{}
\newcommand{\old}[1]{}

\title{Geometric Variants of the Gale--Berlekamp Switching Game}

\titlerunning{Geometric Variants of the Gale--Berlekamp Switching Game}

\author{Adrian Dumitrescu}
{Algoresearch L.L.C., Milwaukee, WI, USA}
{ad.dumitrescu@algoresearch.org}
{0000-0002-1118-0321}
{}

\authorrunning{Adrian Dumitrescu}

\Copyright{Adrian Dumitrescu}

\funding{Work on this paper was supported by ERC Advanced Grant no. 882971  ``GeoScape."}



\begin{document}
   
\maketitle

\begin{abstract}
  The Gale-Berlekamp switching game is played on the following device:
  $G_n=\{1,2,\ldots,n\} \times \{1,2,\ldots,n\}$
  is an $n \times n$ array of lights is controlled by $2n$ switches, one for each row or column.
  Given an (arbitrary) initial configuration of the board, the objective is to have as many lights on
  as possible.
  Denoting the maximum difference (discrepancy) 
  between the number of lights that are on minus the number of lights that are off
  by $F(n)$, it is known (Brown and Spencer, 1971) that
  $F(n)= \Theta(n^{3/2})$, and more precisely, that
  $F(n) \geq \left( 1+ o(1) \right) \sqrt{\frac{2}{\pi}} n^{3/2} \approx 0.797 \ldots n^{3/2}$. 
  Here we extend the game to other playing boards. For example:
  
  (i)~For any constant $c>1$,  if $c n$ switches are conveniently chosen, then the maximum discrepancy
  for the square board is $\Omega(n^{3/2})$.
From the other direction, suppose we fix any set of $a$ column switches, $b$ row switches,
where $a \geq b$ and $a+b=n$. Then the maximum discrepancy is at most $-b (n-b)$.
  
(ii)   A board $H \subset \{1,\ldots,n\}^2$, with area $A=|H|$, 
  is \emph{dense} if $A \geq c (u+v)^2$, for some constant $c>0$, where
  $u= |\{x \colon (x,y) \in H\}|$ and  $v=|\{y \colon (x,y) \in H\}|$. 
  For a dense board of area $A$, we show that the maximum discrepancy is  $\Theta(A^{3/4})$.
  This result is a generalization of the Brown and Spencer result for the original game.

(iii)  If $H$ consists of the elements of $G_n$ below the hyperbola $xy=n$, then its maximum discrepancy
  is $\Omega(n)$ and $O(n (\log n)^{1/2})$. 

  \end{abstract}

\section{Introduction} \label{sec:intro}

Consider a square $n \times n$ array of lights. Each light has two possible states, on and off.
To each row and to each column of the array there is a switch. Turning a switch changes the state of
each light in that row or column. Given an initial state of the board, \ie, a certain on or off position
for each light in the array, the goal is to turn on as many lights as possible.
Equivalently, the number of lights that are on minus the number of lights that are off
should be as large as possible when starting from the initial configuration. See Fig~\ref{fig:boards}.

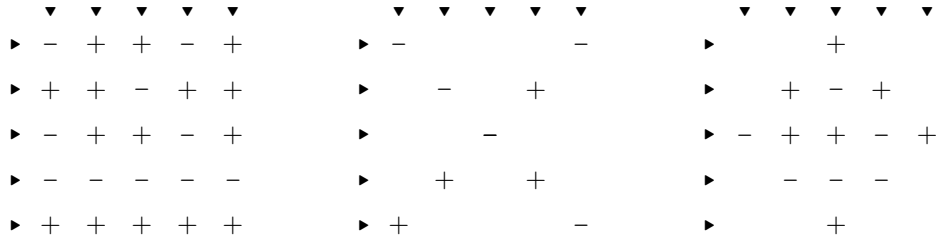
\begin{figure}[htbp]
 \centering 
 \begin{tikzpicture} [scale=0.6]
\draw (0.5,5.25) node{$\filledtriangledown$};
\draw (1.5,5.25) node{$\filledtriangledown$};
\draw (2.5,5.25) node{$\filledtriangledown$};
\draw (3.5,5.25) node{$\filledtriangledown$};
\draw (4.5,5.25) node{$\filledtriangledown$};
\draw (-0.3,0.5) node{$\filledtriangleright$};
\draw (-0.3,1.5) node{$\filledtriangleright$};
\draw (-0.3,2.5) node{$\filledtriangleright$};
\draw (-0.3,3.5) node{$\filledtriangleright$};
\draw (-0.3,4.5) node{$\filledtriangleright$};
\draw (0.5,0.5) node{+};
\draw (1.5,0.5) node{+};
\draw (2.5,0.5) node{+};
\draw (3.5,0.5) node{+};
\draw (4.5,0.5) node{+};
\draw (0.5,1.5) node{$\minus$};
\draw (1.5,1.5) node{$\minus$};
\draw (2.5,1.5) node{$\minus$};
\draw (3.5,1.5) node{$\minus$};
\draw (4.5,1.5) node{$\minus$};
\draw (0.5,2.5) node{$\minus$};
\draw (1.5,2.5) node{+};
\draw (2.5,2.5) node{+};
\draw (3.5,2.5) node{$\minus$};
\draw (4.5,2.5) node{+};
\draw (0.5,3.5) node{+};
\draw (1.5,3.5) node{+};
\draw (2.5,3.5) node{$\minus$};
\draw (3.5,3.5) node{+};
\draw (4.5,3.5) node{+};
\draw (0.5,4.5) node{$\minus$};
\draw (1.5,4.5) node{+};
\draw (2.5,4.5) node{+};
\draw (3.5,4.5) node{$\minus$};
\draw (4.5,4.5) node{+};
\end{tikzpicture}
 \hspace{10mm}
 \begin{tikzpicture} [scale=0.6]
\draw (0.5,5.25) node{$\filledtriangledown$};
\draw (1.5,5.25) node{$\filledtriangledown$};
\draw (2.5,5.25) node{$\filledtriangledown$};
\draw (3.5,5.25) node{$\filledtriangledown$};
\draw (4.5,5.25) node{$\filledtriangledown$};
\draw (-0.3,0.5) node{$\filledtriangleright$};
\draw (-0.3,1.5) node{$\filledtriangleright$};
\draw (-0.3,2.5) node{$\filledtriangleright$};
\draw (-0.3,3.5) node{$\filledtriangleright$};
\draw (-0.3,4.5) node{$\filledtriangleright$};
\draw (0.5,0.5) node{+};
\draw (1.5,1.5) node{+};
\draw (2.5,2.5) node{$\minus$};
\draw (3.5,3.5) node{+};
\draw (4.5,4.5) node{$\minus$};
\draw (0.5,4.5) node{$\minus$};
\draw (1.5,3.5) node{$\minus$};
\draw (2.5,2.5) node{$\minus$};
\draw (3.5,1.5) node{+};
\draw (4.5,0.5) node{$\minus$};

\end{tikzpicture}
  \hspace{10mm}
 \begin{tikzpicture} [scale=0.6]
\draw (0.5,5.25) node{$\filledtriangledown$};
\draw (1.5,5.25) node{$\filledtriangledown$};
\draw (2.5,5.25) node{$\filledtriangledown$};
\draw (3.5,5.25) node{$\filledtriangledown$};
\draw (4.5,5.25) node{$\filledtriangledown$};
\draw (-0.3,0.5) node{$\filledtriangleright$};
\draw (-0.3,1.5) node{$\filledtriangleright$};
\draw (-0.3,2.5) node{$\filledtriangleright$};
\draw (-0.3,3.5) node{$\filledtriangleright$};
\draw (-0.3,4.5) node{$\filledtriangleright$};
\draw (2.5,0.5) node{+};
\draw (1.5,1.5) node{$\minus$};
\draw (2.5,1.5) node{$\minus$};
\draw (3.5,1.5) node{$\minus$};
\draw (0.5,2.5) node{$\minus$};
\draw (1.5,2.5) node{+};
\draw (2.5,2.5) node{+};
\draw (3.5,2.5) node{$\minus$};
\draw (4.5,2.5) node{+};
\draw (1.5,3.5) node{+};
\draw (2.5,3.5) node{$\minus$};
\draw (3.5,3.5) node{+};
\draw (2.5,4.5) node{+};
\end{tikzpicture}
 \caption{Left: a square board whose signed discrepancy is $5$.
   Center: an $X$-shaped board whose signed discrepancy is $-1$; its maximum discrepancy is $5$.
   Right: a rotated square board.}
 \label{fig:boards}
\end{figure}

While at Bell Telephone Laboratories (Murray Hill, NJ) in the late 1960s, Edwyn Berlekamp constructed
a physical instance of this game for $n=10$ and it was a fixture in the tea room of the
Mathematics Department for some years. David Gale is also credited to have invented the game independently,
some time prior to 1971~\cite{BS71}. 

By a \emph{configuration} we shall mean the state of the machine, \ie, precisely which lights are on and off.  
Following Beck and Spencer~\cite{BS83}, define the \emph{signed discrepancy} of a configuration
as the number of lights that are on minus the number of lights that are off.
Equivalently this is just the algebraic sum of the weights $+1$ and $-1$ corresponding to on and off lights,
respectively.

Let the initial configuration be given by an $n \times n$ matrix $A =(a_{ij})$ with all $a_{ij}=\pm 1$,
where $a_{ij}= +1$ if the light in position $(i,j)$ of the board is on; otherwise $a_{ij}= -1$.
Let $G_n=\{1,2,\ldots,n\} \times \{1,2,\ldots,n\}$ be a square piece of the integer lattice,
representing the \emph{configuration}, or \emph{board} of the game. 
Set $x_i=-1$ if the switch for the $i$th row is pulled, otherwise $x_i=+1$.
Set $y_j=-1$ if the switch for the $j$th column is pulled, otherwise $y_j=+1$.   
Some reflection indicates that the order in which switches are pulled is irrelevant, it only matters
which ones are pulled. It is easily seen that the light in position $(i,j)$ has value $a_{ij} x_i y_j$,
as affected by the  line switches $x_i$ and $y_j$, whence the discrepancy is given by
\begin{equation} \label{eq:f(H)}
  f(H) = N_+ - N_- = \sum_{i=1}^n \sum_{j=1}^n  a_{ij} x_i y_j.
\end{equation}
Set now
\begin{equation} \label{eq:F(n)}
  F(n) = \min_{a_{ij}} \max_{x_i, y_j} \sum_{i=1}^n \sum_{j=1}^n  a_{ij} x_i y_j,
\end{equation}
which represents the \emph{maximum (signed) discrepancy} of the square board in the worst case.

Solving the game can also be interpreted as determining the covering radius of the so-called
light-bulb code of length $n^2$, as explained in~\cite{Sl87}. When Berlekamp introduced his game,
it was quickly seen that a brute force computational approach is infeasible even for $n=10$. 
However, by using symmetry, case elimination, and computer help, Fishburn and Sloane~\cite{FS89} 
found optimal solutions up to $n=10$; the solution for $n=10$ was subsequently corrected by
Carlson and Stolarski~\cite{CS04} who also verified the previous solutions for $n \leq 9$. 
The first few values of $F(n)$ are given in Table~\ref{tab:F(n)}.

\begin{table}[ht]
  \caption{The maximum discrepancy $F(n)$ for a square board of size $n$, for $n=2,\ldots,10$.}
\begin{center}
\begin{tabular}{|c||c|c|c|c|c|c|c|c|c|}
\hline
$n$ & $2$ & $3$ & $4$ & $5$ & $6$ & $7$ & $8$ & $9$ & $10$ \\
\hline \hline
   & $2$ & $5$ & $8$ & $11$ & $14$ & $17$ & $20$ & $27$ & $30$ \\
\hline
\end{tabular}
\end{center}
\label{tab:F(n)}
\end{table}

\subparagraph{Extending the game.}
The Gale-Berlekamp switching game can be extended as follows.
Let $H \subseteq G_n$ represent the \emph{configuration}, or \emph{board} of the game. 
The original version of the game corresponds to $H_0=G_n$.
Note that $|H_0|=n^2$, and $|H| \leq n^2$, for an arbitrary board. 
Given $H$, let $H_x$ and $H_y$ defined below represent the $x$- and $y$-\emph{projection}
of $H$, respectively; let $u$ and $v$ denote the sizes of these projections.
\begin{equation}  \label{eq:proj}
  H_x= \{x \colon (x,y) \in H\}, \ \  H_y = \{y \colon (x,y) \in H\}, \text{ and } u=|H_x|, v=|H_y|.
\end{equation}

Given an arbitrary board $H$, set
\begin{equation} \label{eq:F(H)}
  F(H) = \min_{a_{ij}} \max_{x_i, y_j} \sum_{i \in H_x} \sum_{j \in H_y}  a_{ij} x_i y_j,
\end{equation}
which represents the \emph{maximum discrepancy} of $H$.

\subparagraph{Our results.}
We first show how the game can be extended to arbitrary boards under the same principles.
We start by analyzing various modifications of the square board with respect to increasing
or decreasing the number of line switches.

Suppose now that additional line switches are provided: a \emph{line switch} acts
in the same way as a row or column switch.
With roughly $tn$ line switches (for some $t \in \NN$), the maximum discrepancy of the
square board is shown to be $\Theta(n^{3/2} t^{1/2})$ (Theorem~\ref{thm:add}).
In particular, for $t \sim n$, the maximum discrepancy is quadratic, \ie, the most one can hope for.

For any constant $c>1$, if $c n$ row and column switches are conveniently chosen, then
the maximum discrepancy is $\Omega(n^{3/2})$ (Theorem~\ref{thm:remove}).
From the other direction,  suppose we fix any set of $a$ column switches,
$b$ row switches, where $a \geq b$ and $a+b=n$. Then the maximum discrepancy
is at most $-b (n-b)$ (Theorem~\ref{thm:remove}). These bounds are essentially the best possible.

Further, we obtain a tight asymptotic bound on the maximum discrepancy for dense boards
(Theorem~\ref{thm:dense}).   This result is a generalization of the Brown and Spencer result
for the original game.
We then derive more precise bounds in  two specific instances
featuring an inscribed square and an inscribed disk.
The bounds in Theorems~\ref{thm:45} and~\ref{thm:disk} are special
cases of the bound in Theorem~\ref{thm:dense}.

In Section~\ref{sec:other} we consider two other boards that extend the above framework.
Firstly, we analyze a planar hyperbolic board that is not dense (Theorem~\ref{thm:hyperbola})
and so expectedly has a much higher signed discrepancy; here the resulting bounds differ but are
very close to each other.
Secondly, we derive a tight asymptotic bound for a $3$-dimensional version (Theorem~\ref{thm:cube}):
if $H=\{1,\ldots,n\}^3$ is the standard cubic piece of the integer lattice,
then its maximum discrepancy is $\Theta(n^{5/2})$ in the model with $3n^2$ line switches,
whereas it is known to be $\Theta(n^{2})$ in the model with $3n$ plane switches~\cite{AP19}.

\smallskip
It should be noted that overly precise board specifications w.r.t. the boundary structure
are irrelevant for the asymptotic results obtained in Theorems~\ref{thm:45},~\ref{thm:disk},
and~\ref{thm:hyperbola}. This is due to the boundary effects being negligible for the boards
considered there.

\subparagraph{Related work.}
Let $\Sigma=\{-1,1\}$. Given a matrix $M \in \Sigma^{n \times n}$, the problem of finding vectors
$x, y \in \Sigma^n$ maximizing $x^T M y$, \ie, solving the Gale-Berlekamp Switching Game, 
was shown to be $\NP$-hard by Roth and Viswanathan~\cite{RV08}. 
Karpinski and Schudy~\cite{KS09} designed a linear time approximation scheme for the
game and generalized it to a wider class of minimization problems
including the Nearest Codeword Problem (NCP) and Unique Games Problem (UGP).

Pellegrino, Silva, and Teixeira studied a continuous version of the game in which vectors replace
light bulbs and knobs substitute the discrete switches used in the original problem~\cite{PST23}.
A light switching game based on permutations with light bulbs in an
$n \times n$ formation, having $k$ different intensities has been considered
and analyzed by Brualdi and Meyer~\cite{BM15}. 

The dual problem of minimizing the discrepancy has been also considered.
About $60$ years ago, Leo Moser made the following conjecture (see the account in~\cite{BS83}):
Given any $n \times n$ matrix $A =(a_{ij})$ with all $a_{ij}=\pm 1$,
there exist $x_1,\ldots,x_n$, $y_1,\ldots,y_n = \pm 1$ so that, setting
\[ D = \left | \sum_{i=1}^n \sum_{j=1}^n x_i y_j a_{ij} \right |, \]
$D=1$, when $n$ is odd, and $D \leq 2$, when $n$ is even.
In 1969, Koml\'os and M. Sulyok~\cite{KS69} succeeded in proving this result for all
sufficiently large $n$.  The full conjecture was proved by Beck and Spencer in 1983~\cite{BS83}.  
Balanced colorings for lattice points with respect to axis-parallel lines have been considered by
Akiyama and Urrutia~\cite{AU90}.

\section{Preliminaries} \label{sec:prelim}

\subparagraph{A warm-up example.} 
Consider an $X$-shaped board like the one depicted in Fig.~\ref{fig:boards}\,(center),
and suppose that $n$ is odd, say, $n=2k+1$. Observe that the board is decomposable into 
$k$ nested axis-aligned concentric $4$-cycles (as squares) plus the center, and each of these
are controlled independently.
In each of these cycles, a discrepancy of $3-1=2$ can be achieved, so altogether (with the center),
a discrepancy of $2k+1=n$ can be achieved. Indeed, we show that at most two line switches suffice to
bring any $4$-cycle into $+1,+1,+1,+1$, or $+1,+1,+1,+-1$.

If a cycle is of the form $-1,-1,-1,-1$, two switches of the parallel lines
containing its vertices yield $+1,+1,+1,+1$.
If a cycle is of the form $+1,+1,+1,-1$, it is left unchanged;
moreover, this is the best possible, \ie, the $-1$ cannot be eliminated.
If it is of the form $+1,-1,+1,-1$, switching the last two elements for example, 
yields $+1,-1,-1,+1$ and switching the two middle elements finally yields $+1,+1,+1,+1$.

The case of even $n$ is solved using the same procedure.
It is easy to see that the upper estimate of $2$ per cycle is the best possible
(by the previous observation). We have thus shown the following:

\begin{proposition} \label{prop:X}
Let $H$ be an $X$-shaped board inscribed in $G_n$ (with $u=v=n$).
Then its maximum discrepancy is $F(H)=n$. 
\end{proposition}

\subparagraph{The original version of the game.}
We now recall the square board $G_n$ studied by Brown and Spencer~\cite{BS71}
and revisited by Spencer~\cite{Spe94}; as we will be relying and building upon the techniques
employed there.
For the upper bound, the authors of~\cite{BS71} used Hadamard matrices and employed a method of
Moon and Moser~\cite{MM66}. A subsequent probabilistic approach was offered by Spencer~\cite{Spe94}.
For the lower bound, the following well-known deviation inequality is used:
see~\cite[Chap.~5]{MU17}, or~\cite[Lecture~4]{Spe94}.

\begin{theorem} \label{thm:dev-S_n}
  Let $S_n = X_1 + \cdots + X_n$, where
  \[ \Pr \left[X_i =+1 \right] = \Pr \left[X_i =-1 \right] = \frac12 \]
  and the $X_i$ are mutually independent. Then for any $\lambda>0$
  \begin{equation} \label{eq:dev-S_n}
    \Pr \left[ S_n > \lambda \right] < \exp(-\lambda^2/2n). 
  \end{equation}
 \end{theorem}

\medskip
Further, it is known~\cite[Lecture 6]{Spe94} that
\[ \E[|S_n|] = n \cdot 2^{1-n} {n-1 \choose \lfloor (n-1)/2 \rfloor}. \]
Using Stirling's approximation for the factorial yields
\begin{equation} \label{eq:asy-S_n}
\E[|S_n|] =\left( 1+ o(1) \right)  \sqrt{\frac{2}{\pi}} n^{1/2}.
\end{equation}

Our starting point for our results in Sections~\ref{sec:dense} and~\ref{sec:special}
is a proof for the original version of the game with a square board.
In particular, we refer to the proof in~\cite[Lecture 6]{Spe94}. We include it here
for completeness.

\begin{theorem} \label{thm:square}  {\rm (Brown and Spencer~\cite{BS71})}.
  Let $H=G_n$.
  Then the maximum discrepancy $F(H)$ is $\Theta(n^{3/2})$, and more precisely,
\begin{align} 
F(H) &\leq c n^{3/2}, \text{ where } c = 2 (\ln{2})^{1/2}, \text{ and }  \label{eq:ub-square} \\
F(H) &\geq \left( 1+ o(1) \right) \sqrt{\frac{2}{\pi}} n^{3/2}. \label{eq:lb-square} 
\end{align}
\end{theorem}
\begin{proof}
  \emph{Upper bound.}
 We employ the basic probabilistic method~\cite{AS16}.
  Let $a_{ij}=\pm1$ be chosen independently and at random with equal probability.
  For any fixed choice $x_i,y_j$, the terms $a_{ij} x_i y_j$ are mutually independent
  and equal to $\pm1$ with the same probability.  This means that 
  $\sum_{i=1}^n \sum_{j=1}^n  a_{ij} x_i y_j$ has distribution $S_{n^2}$.
  Recall that $c = 2 (\ln{2})^{1/2}$ and set $\lambda = c n^{3/2}$.
  Applying the deviation inequality~\eqref{eq:dev-S_n} yields
  \[ \Pr \left[ \sum_{i=1}^n \sum_{j=1}^n  a_{ij} x_i y_j > \lambda \right]
  < \exp(-\lambda^2/2n^2) = 2^{-2n}. \]
  There are $2^{2n}$ choices for $x_i,y_j$, thus by the union bound we obtain
  \[ \Pr \left[ \sum_{i=1}^n \sum_{j=1}^n  a_{ij} x_i y_j > \lambda \text{ for some } x_i, y_j \right]
  < 2^{2n} \cdot 2^{-2n} =1. \]
  Consequently, there exists a matrix
  $[a_{ij}]$ such that $\sum_{i=1}^n \sum_{j=1}^n  a_{ij} x_i y_j  \leq \lambda = c n^{3/2}$ for all $x_i,y_j$.

  \smallskip
  \emph{Lower bound.} 
Let $A=[a_{ij}]$ be fixed. Some of the switches are set in a probabilistic way, whereas
others are set deterministically, as follows. Let the column switches $y_j=\pm1$ be selected at random and let
\[ R_i = \sum_{j=1}^n  a_{ij} y_j \]
be the ``scrambled'' $i$th row, namely the $i$th row after the switches $y_j$.
For a fixed $i$, the row $(a_{i1} y_1, a_{i2} y_2,\ldots,  a_{in} y_n)$ can be any of the
$2^n$ possibilities, equally likely. Hence the distribution of the row is uniform irrespective of the
row's initial value, and $R_i$ has distribution $S_n$. It follows that
\[ \E \left[ |R_i| \right] = \left( 1+ o(1) \right) \sqrt{\frac{2}{\pi}} n^{1/2}. \]
Note that by the linearity of expectation, we have
\[ \E \left[ \sum_{i=1}^n |R_i| \right] = \sum_{i=1}^n  \E \left[ |R_i| \right]. \]

Suppose that we already have obtained (by our random choices) these inequalities
for each row of the matrix, as guaranteed by their expectations.  
Set now the row switches $x_i$ such that $x_i R_i = |R_i|$ for every $i$, \ie, such that
the row discrepancies result in a cumulative gain. Then
\[ \sum_{i=1}^n \sum_{j=1}^n  a_{ij} x_i y_j   = \sum_{i=1}^n x_i R_i = \sum_{i=1}^n |R_i|
= \left( 1+ o(1) \right) \sqrt{\frac{2}{\pi}}  n^{3/2}.     \qedhere  \] 
\end{proof}

\section{Modifications of the square board} \label{sec:modifications}

In this section we address questions of the following nature:

\begin{enumerate}  \itemsep 2pt

\item What is the maximum discrepancy achievable for the square board,
  when adding more switches?

\item What is the maximum discrepancy achievable under budget constraints?
  For instance, with only $n$ line switches?

\end{enumerate}

\subsection{Upgrades: adding switches} \label{sec:add}

Presumably the discrepancy can be increased by adding more switches to the board.
We start by showing that sometimes this is the case.
For instance, let $H$ be the square board $G_n$ equipped with $2n-1$ switches
along the diagonal lines of unit slope and $n$ column switches; $3n-1$ switches in total. 

In the first step, select the column switches $y_i=\pm1$ at random and observe that
each diagonal line $\ell$ has distribution $S_{|\ell \cap G_n|}$ regardless of the initial
values of the matrix entries. As in the proof of Theorem~\ref{thm:square}, we obtain
\[ F(H) \geq \left( 1+ o(1) \right) \sqrt{\frac{2}{\pi}} \left( 2\sum_{i=1}^n \sqrt{i} - \sqrt{n} \right)=
\frac43 \sqrt{\frac{2}{\pi}}  \left( 1+ o(1) \right) n^{3/2}. \]
This represents an increase by a factor of $4/3$ from the original version. 
For the upper bound, one gets an asymptotically matching bound by requiring
$\exp\left(\frac{-\lambda^2}{2A}\right) = 2^{3n}$.
Adding more switches allows for a proportional increase in discrepancy, as shown below.

\begin{theorem} \label{thm:add}  
Let $t <n$ be a positive integer and $H$ be the square board $G_n$ with
$n$ switches on the horizontal lines, and a switch on each line of slope $t$
that is incident to some point of $G_n$, see Fig.~\ref{fig:t-lines}. 
Then the maximum discrepancy of $H$ is $\Theta(n^{3/2} t^{1/2})$.
\end{theorem}

\begin{figure}[htbp]
\centering
\includegraphics[scale=0.75]{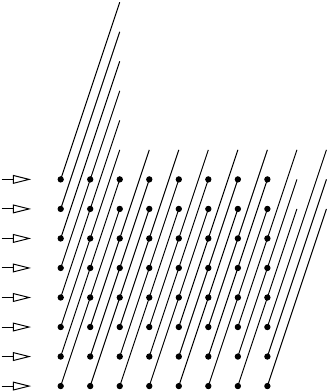}
\caption{Switches on slanted lines of slope $t=3$ and horizontal lines.}
\label{fig:t-lines}
\end{figure}
\begin{proof}
  There are $nt+n-t$ slanted lines and so there are $\Theta(nt)$ switches on these lines.
  Indeed, a point of $G_n$ is the lower left end of a slanted line if and only if it belongs
  to the lowest $t$ rows  or the first column. 
In the first step, select the row switches $x_i=\pm1$ at random and observe that
each slanted line $\ell$ has distribution $S_{|\ell \cap G_n|}$.
In the second step, select the switches on the slanted lines so as to add up
their absolute values and obtain
\[ F(H) = \Omega(nt) \cdot \sqrt{\frac{n}{t}} = \Omega(n^{3/2} t^{1/2}). \]

For the upper bound, requiring $\exp\left(\frac{-\lambda^2}{2n^2}\right) = 2^{-nt}$,
yields a matching asymptotic upper bound, namely $F(H) = O(n^{3/2} t^{1/2})$.
\end{proof}

\subsection{Downgrades: removing switches} \label{sec:remove}

Suppose that we have $a \leq n$ column switches and $b \leq n$ row switches, and assume, as we may,
that $a+b <2n$ and $a \geq b$. We first show that the number of switches, $2n$, in the standard version
is tight in a certain sense: if, say, $1.99 n$ switches are chosen in an adversarial way, then the maximum
discrepancy is at most $- cn^2$, for some constant $c>0$, that is, almost the worst it can be.   
On the other hand, if, say, $1.01 n$ switches are conveniently chosen,  then the maximum
discrepancy is $\Omega(n^{3/2})$,  that is, almost the best it can be. 
Finally, if only $n$ switches are allowed, then the maximum discrepancy is zero, at best.
Consider for instance the board with $n$ column switches; then the maximum discrepancy is zero.
Indeed, one can always set the switches to make each column non-negative. On the other hand,
for a zero-sum matrix consisting of zero-sum columns, this is the best possible. 
The scenarios we considered are summarized in the following.

\begin{theorem} \label{thm:remove}  
  Let $H$ be the square board $G_n$ with $a$ column switches and $b \leq a$ row switches,
  see Fig.~\ref{fig:ab}\,(left), and $\delta>0$ be a constant.
  
\begin{enumerate}  [(i)] \itemsep 2pt

\item If $a=n$, and $b \geq \delta n^{1/2}$, for a sufficiently large $\delta$, then the maximum discrepancy
  of $H$ is $\Omega(b n^{1/2}) =\Omega(n)$. In particular, if $b \geq \delta n$, then 
  $F(H) \geq c(\delta) n^{3/2}$, where $c(\delta)>0$. 

\item If $a=b=(1-\delta)n$, then the maximum discrepancy of $H$ is
  $F(H) \leq -c(\delta) n^2$, where $c(\delta)>0$. 

\item Fix any set of $n$ switches, where $a \geq b$ (and $a+b=n$). Then $F(H) \leq -b (n-b)$.
  In particular, if $b \geq \delta n$, then $F(H) \leq -c(\delta) n^{2}$, where $c(\delta)>0$. 

\end{enumerate}

Apart from the values of the constants, these bounds cannot be improved.
\end{theorem}

\begin{figure}[htbp]
\centering
\includegraphics[scale=0.8]{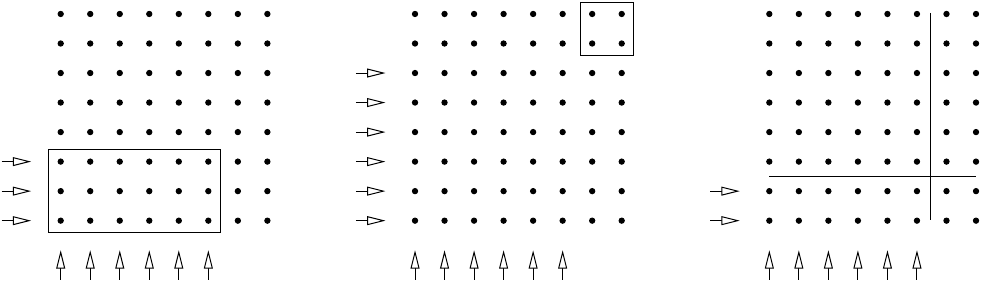}
\caption{A square board with $n=8$. Left: $a=6, b=3$ ($9$ switches).
  Center: $a=b=6$ ($12$ switches). Right: $a=6, b=2$ ($8$ switches).}
\label{fig:ab}
\end{figure}

\begin{proof}
  By symmetry, we may assume that the available switches control the lower left submatrix
  with $a$ columns and $b$ rows, as in Fig.~\ref{fig:ab}.

  \smallskip
  (i) First scramble the rows using the $n$ column switches uniformly at random and then
  add the absolute values over the $b$ bottom rows using the corresponding row switches accordingly.
  The expected absolute value of each row (including the $b$ controlled by switches) is $\sim n^{1/2}$.
  Observe that $\sum a_{ij} x_i y_j$ over the top $n-b$ rows has distribution $S_{n(n-b)}$, thus
  the expected absolute value of this sum is $O(\sqrt{n(n-b)}) = O(n)$. Overall, we obtain
  \[ F(H) = \Omega(b n^{1/2}) - O(n) = \Omega(b n^{1/2}). \]

 \smallskip
 (ii)  Set each element in the top right submatrix $A'$ of $A$ (not controlled by any switches) to $-1$,
 and choose every other element of $A$ uniformly and random from $\{-1,1\}$.
 Refer to Fig.~\ref{fig:ab} \,(center).
  For any fixed choice $x_i,y_j$, the terms $a_{ij} x_i y_j$ are mutually independent
  and equal to $\pm1$ with the same probability.  This means that 
  $\sum a_{ij} x_i y_j$ over $A \setminus A'$ has distribution $S_{(1-\delta^2)n^2}$,
  and a union bound (as in the upper bound argument in the proof of Theorem~\ref{thm:square})
  implies that there exists a matrix with $\sum_{A \setminus A'} a_{ij} x_i y_j = O(n^{3/2})$.  As such,
  \[ \sum_{A} a_{ij} x_i y_j =  \sum_{A \setminus A'} a_{ij} x_i y_j + \sum_{A'} a_{ij} 
  = O(n^{3/2})  -\delta^2 n^2 \leq -c(\delta) n^2, \]
  where $c(\delta)>0$. 

  \smallskip
  (iii) The case $b=0$ was discussed before the theorem. Assume now that $b \geq 1$
  and refer to Fig.~\ref{fig:ab} \,(right).
  Let the upper-left and lower-right submatrices of $A$ be filled with $\pm 1$ in a chessboard pattern,
  and let the other two submatrices of $A$ be filled with $-1$. Initially,  we have $\sum a_{ij} =-2b(n-b)$,
  and after the switches we have $\sum a_{ij} x_i y_j \leq  -2b(n-b) + b(n-b) = -b(n-b)$, since at most
  $b(n-b)$ entries can be affected by two switches, whereas switches in the submatrices with chessboard pattern
  do not influence the result.
\end{proof}

\section{Dense boards} \label{sec:dense}

In this section we obtain a tight asymptotic bound on the maximum discrepancy of a dense board
and thereby generalize Theorem~\ref{thm:square}.

\begin{theorem} \label{thm:dense}
  Let $H$ be a dense board, with $A=\area(H) \geq c (u+v)^2$, for some constant $c>0$. 
  Then its maximum discrepancy satisfies $F(H) \leq c_1 A^{3/4}$, for some constant $c_1>0$ that depends on $c$.
  Apart from the multiplicative constant, the above inequality is the best possible:
  $F(H) \geq c_2 A^{3/4}$, for some constant $c_2>0$ that depends on $c$.
\end{theorem} 
\begin{proof}
 \emph{Upper bound.}
 We employ the basic probabilistic method.
  For $(i,j) \in H$, let $a_{ij}=\pm1$ be chosen independently and at random with equal probability.
  For any fixed choice $x_i,y_j$, the terms $a_{ij} x_i y_j$ are mutually independent
  and equal to $\pm1$ with the same probability.  This means that 
  $\sum_{(i,j) \in H} a_{ij} x_i y_j$ has distribution $S_{A}$. 
  Set $\lambda = c_1 A^{3/4}$, where $c_1 = (2 \ln{2})^{1/2} \cdot c^{-1/4}$.
  Applying the deviation inequality~\eqref{eq:dev-S_n} while using the assumed lower bound on $A$ yields
  \[ \Pr \left[ \sum_{i=1}^n \sum_{j=1}^n  a_{ij} x_i y_j > \lambda \right]
  < \exp(-\lambda^2/2A) = \exp(-c_1^2 A^{1/2}/2) \leq 2^{-(u+v)}. \]
  There are $2^{u+v}$ choices for $x_i,y_j$, thus by the union bound we obtain
  \[ \Pr \left[ \sum_{i=1}^n \sum_{j=1}^n  a_{ij} x_i y_j > \lambda \text{ for some } x_i, y_j \right]
  < 2^{u+v} \cdot 2^{-(u+v)} =1. \]
  Consequently, there exists a matrix $[a_{ij}]$ such that
  \[ \sum_{i=1}^n \sum_{j=1}^n  a_{ij} x_i y_j  \leq \lambda = c_1 A^{3/4}, \text{  for all } x_i,y_j.  \]

\medskip
\emph{Lower bound.}
Let $H$ be any fixed dense board with $A \geq c (u+v)^2$, for some constant $c$.
This implies $A \geq c (u+v)^2 \geq 4 c uv$. 
We show that $F(H) \geq c_2 A^{3/4}$, for some constant $c_2>0$.
 Recall the setting in~\eqref{eq:proj}, where we may assume that $u \geq v$.
 We claim that there exist at least $cu$ rows each with at least $cv$ elements of $H$.
 Assume for contradiction that the above statement does not hold. Then there exist 
at least $(1-c)u$ rows each with fewer than $cv$ elements of $H$.
Since each row has at most $v$ elements of $H$, there are at most 
$cu$ other rows each with at most $v$ elements of $H$. Hence the total number
of elements in $H$ is less than
\[ (1-c) u \cdot cv + c uv=(2c -c^2) uv < 4 c uv, \]
a contradiction; this concludes the proof of the claim.

Some of the switches are set in a probabilistic way, whereas others are set deterministically, as follows.
Let the column switches $y_j=\pm1$ be selected at random and let
\[ R_i = \sum_{j \in H_y} a_{ij} y_j \]
be the ``scrambled'' $i$th row, $i \in H_x$, namely the $i$th row after the switches $y_j$.
For a fixed $i$, the $i$th row can be any of the $2^{v_i}$ possibilities, equally likely, where
$v_i \leq v$ is the number of elements of $H$ in row $i$. 
Hence the distribution of the row is uniform regardless of the
row's initial value, and $R_i$ has distribution $S_{v_i}$. It follows that
\[ \E \left[ |R_i| \right] = \Omega(\sqrt{v_i}). \]
Note that by the linearity of expectation, we have
\[ \E \left[ \sum_{i \in H_x} |R_i| \right] = \sum_{i \in H_x}  \E \left[ |R_i| \right]. \]

Suppose that we already have obtained (by our random choices) these inequalities
for each row of the matrix, as guaranteed by their expectations.  
Set now the row switches $x_i$ such that $x_i R_i = |R_i|$ for every $i$, \ie, such that
the row discrepancies result in a cumulative gain. Then
\[ \sum_{i \in H_x} \sum_{j \in H_y}   a_{ij} x_i y_j   = \sum_{i \in H_x } x_i R_i = \sum_{i \in H_x} |R_i|. \]
By the above claim, there exist at least $cu$ rows each with at least $cv$ elements of $H$.
For each such row, the discrepancy is $\Omega(\sqrt{cv}) = c ^{1/2} \, \Omega(\sqrt{v})$.
Consequently,
\[ F(H) = \Omega( c^{3/2} \cdot u\sqrt{v}) = \Omega( u\sqrt{v}) = \Omega(A^{3/4}), \]
as required. Indeed, the inequality $u \sqrt{v} \geq A^{3/4}$ immediately follows by multiplying the
following three inequalities (where $uv \geq A$ holds by definition, and $u \geq v$ holds by assumption)
and taking the $1/4$th power.
\[ uv \geq A, \ \ uv \geq A, \ \ u^2 \geq uv \geq A. \]
It is easily seen that one can take $c_2 \sim c^{3/2}$.
\end{proof}

\subsection{Two applications} \label{subsec:applications}

In this subsection we consider two specific examples of dense boards in Theorems~\ref{thm:45} and~\ref{thm:disk}.
A~board example for Theorem~\ref{thm:45} appears in Fig.~\ref{fig:boards}\,(right). 

\begin{theorem} \label{thm:45}
  Let $H$ be a maximal square whose sides are parallel to $y=\pm x$ that is inscribed in $G_n$.
  Then its maximum discrepancy $F(H)$ is $\Theta(n^{3/2})$, and more precisely,
\begin{align} 
  F(H) &\geq \left( 1+ o(1) \right) \frac23 \cdot \sqrt{\frac{2}{\pi}} \cdot n^{3/2}, \text{ and }\\ 
F(H) &\leq c n^{3/2}, \text{ where } c = (2\ln{2})^{1/2}. 
\end{align}
\end{theorem}
\begin{proof}
\emph{Lower bound.}
  We proceed as in the proof of Theorem~\ref{thm:square}.
  Let $k=n/2$. Note that
 \[ \int_0^k \sqrt{x} \intd x = \left. \frac{x^{3/2}}{3/2} \right \vert_0^k = \frac23 \, k^{3/2}. \]
By the specifics of our board, as in the proof of Theorem~\ref{thm:dense}, we obtain
\begin{align*}
\E \left[ \sum_{i=1}^n |R_i| \right]  &=\left(1+ o(1) \right) \sqrt{\frac{2}{\pi}} \cdot 2 \sum_{i=1}^k \sqrt{2i} 
= \left(1+ o(1) \right) \sqrt{\frac{2}{\pi}} \cdot 2 \sqrt2 \, \int_0^k \sqrt{x} \intd x \\
&= \left(1+ o(1) \right) \sqrt{\frac{2}{\pi}} \cdot 2 \sqrt2 \cdot \frac23 \cdot \frac{n^{3/2}}{2 \sqrt2}
= \left(1+ o(1) \right) \frac23 \cdot \sqrt{\frac{2}{\pi}} \cdot n^{3/2}.
\end{align*}

\smallskip
\noindent \emph{Upper bound.}
  An upper bound $F(H) =O(n^{3/2})$ follows from Theorem~\ref{thm:dense}. Indeed,
  we have $A=(1+o(1)) n^2/2$, whence $A^{3/4} =\Theta(n^{3/2})$. More precisely, the inequality
  \[  \exp(-\lambda^2/2A) = \exp(-c_1^2 A^{1/2}/2) \leq 2^{-(u+v)}, \]
  where $u=v=n$, is satisfied by setting $\exp(c_1^2 A^{1/2}/2) = 2^{2n}$, which yields 
$c_1= 2^{5/4} (\ln 2)^{1/2}$ and $c = (2 \ln{2})^{1/2}$. 
\end{proof}

For the next result, recall that the Gamma function is given by the improper integral
\begin{equation} \label{eq:Gamma}
  \Gamma(s) = \int_0^\infty t^{s-1} e^{-t} \intd t.
\end{equation}
It is known that $\Gamma(s+1)= s \, \Gamma(s)$.
In particular, $\Gamma(-0.25)= -4 \, \Gamma(0.75)$, and
 $\Gamma(1.75)= 0.75 \, \Gamma(0.75)$, where $\Gamma(0.75)= 1.225\ldots$.
For every positive integer $n$, $\Gamma(n)=(n-1)!$.

\begin{theorem} \label{thm:disk}
  Let $H$ be a maximal disk that is inscribed in $G_n$.
  Then its maximum discrepancy $F(H)$ is $\Theta(n^{3/2})$, and more precisely,
\begin{align} 
  F(H) &\geq \left(1+ o(1) \right) \left( \frac{\pi}{|\Gamma(-0.25)| \cdot \Gamma(1.75)} \right) n^{3/2},
\text{ and }\\ 
F(H) &\leq c n^{3/2}, \text{ where } c= \pi^{1/4}  (\ln 2)^{1/2}.
\end{align}
\end{theorem}
\begin{proof}
\emph{Lower bound.}
We proceed as in the proofs of Theorem~\ref{thm:square} and~\ref{thm:45}. 
  Let $k=n/2$. A~standard computational system such as \emph{WolframAlpha} shows that 
  \begin{align*}
    \int_0^k (k^2 -x^2)^{1/4}  \intd x  &= \int_0^{\pi/2} (k^2 \cos^2 t)^{1/4} k \cos t \intd t \\
   &= k^{3/2} \int_0^{\pi/2} \cos^{3/2} t \intd t =
    \frac{\pi^{3/2} k^{3/2}}{2^{1/2} |\Gamma(-0.25)| \cdot \Gamma(1.75)}.
  \end{align*}
The first two steps in the above chain are the result of a change of variable: $x=k \sin t$.
Note that $\Gamma(-0.25)=-4.901\ldots<0$ and $\Gamma(1.75)=0.919\ldots$.

By the specifics of our board, as in the proof of Theorem~\ref{thm:dense}, we obtain
\begin{align*}
\E \left[ \sum_{i=1}^n |R_i| \right] &=\left(1+ o(1) \right) \sqrt{\frac{2}{\pi}} \cdot 2 \sqrt2 \cdot \sum_{i=1}^k (k^2 -i^2)^{1/4} \\
&= \left(1+ o(1) \right) \sqrt{\frac{2}{\pi}} \cdot 2 \sqrt2 \cdot \int_0^k (k^2 - x^2)^{1/4} \intd x \\
&= \left(1+ o(1) \right) \sqrt{\frac{2}{\pi}} \cdot 2 \sqrt2 \cdot
\frac{\pi^{3/2} k^{3/2}}{2^{1/2} |\Gamma(-0.25)| \cdot \Gamma(1.75)} \\
&= \left(1+ o(1) \right) (2 \sqrt2) \cdot \frac{ \pi \, n^{3/2}}{(2 \sqrt2) |\Gamma(-0.25)| \cdot \Gamma(1.75)}\\
&= \left(1+ o(1) \right) \cdot \frac{\pi}{|\Gamma(-0.25)| \cdot \Gamma(1.75)} \cdot n^{3/2}. 
\end{align*}

\smallskip
\noindent \emph{Upper bound.}
We have $A=(1+o(1)) \pi n^2/4$ and $u=v=n$. For satisfying 
\[  \exp(-\lambda^2/2A) = \exp(-c_1^2 A^{1/2}/2) \leq 2^{-(u+v)}, \]
 we set $\exp(c_1^2 A^{1/2}/2) = 2^{2n}$, which yields
$c_1 = 2^{3/2} \pi^{-1/2} (\ln 2)^{1/2}$ and $c = \pi^{1/4}  (\ln 2)^{1/2} $.
\end{proof}

\section{Other boards} \label{sec:other}

\subsection{A special planar board} \label{sec:special}

In this section we consider an interesting case where applying the previous techniques does not yield a tight bound
on the signed discrepancy. However, the resulting bounds are only a $O((\log n)^{1/2})$ factor away from each other.

\begin{theorem} \label{thm:hyperbola}
  Let $H$ consist of the elements of $G_n$ on or below the hyperbola $xy=n$.
  Then its maximum discrepancy is $\Omega(n)$ and $O(n (\log n)^{1/2})$. 
\end{theorem}
\begin{proof}
\emph{Lower bound.}
We proceed as in the proofs of Theorem~\ref{thm:square} and~\ref{thm:45}.
We will need the following well-known estimate:
\begin{equation} \label{eq:int-1/2}
  \sum_{i=1}^n i^{-1/2} = \Theta \left( \int_1^n  x^{-1/2} \intd x \right) = \Theta (n^{1/2}).
\end{equation}
By using~\eqref{eq:int-1/2} we obtain
\begin{equation*}
\E \left[ \sum_{i=1}^n |R_i| \right] = \Omega \left( \sum_{i=1}^n \sqrt{\frac{n}{i}} \right) 
= \sqrt{n} \cdot \Omega \left(\sum_{i=1}^n i^{-1/2} \right) = \Omega(n).
\end{equation*}

(An alternative deterministic lower bound solution is as follows: 
Turn every entry into $+1$ in the row with $n$ entries, and turn each other row non-negative.
It yields that the maximum discrepancy is $\geq n$.)

\medskip
\noindent \emph{Upper bound.}
We have
\[ A= \Theta \left( \sum_{i=1}^n \frac{n}{i} \right)  = \Theta (n \log{n}), 
\text{ and } u=v=n.  \]
For satisfying $\exp(-\lambda^2/2A) \leq 2^{-(u+v)} = 2^{-2n}$,
set $\lambda \sim n (\log{n})^{1/2}$, completing the proof.
\end{proof}

 \subsection{The cubic board} \label{sec:cube}

Consider the $3$-dimensional cubic board  $H= \{1,\ldots,n\}^3$ equipped with $3n^2$ line switches:
\begin{align*}
\ell_{ij} \ \colon \ x_i &=i, \ y_j =j, \ \ i,j=1,\ldots,n;  \\
\ell_{ik} \ \colon \ x_i &=i, \ z_k =k, \ \ i,k=1,\ldots,n;  \text{ and } \\
\ell_{jk} \ \colon \ y_j &=j, \ z_k =k, \ \ j,k=1,\ldots,n.
\end{align*}

\begin{theorem} \label{thm:cube} 
  Let $H= \{1,\ldots,n\}^3$.
  Then the maximum discrepancy $F(H)$ is $\Theta(n^{5/2})$, and more precisely,
\begin{align} 
F(H) &\leq c n^{5/2}, \text{ where } c = (6 \ln{2})^{1/2},  \text{ and }  \label{eq:ub-cube} \\
F(H) &\geq \left( 1+ o(1) \right) \sqrt{\frac{2}{\pi}}  \, n^{5/2}.  \label{eq:lb-cube} 
\end{align}
\end{theorem}
\begin{proof}
\emph{Upper bound.}
 We employ the basic probabilistic method~\cite{AS16}.
  Let $a_{ijk}=\pm1$ be chosen independently and at random with equal probability.
  For any fixed choice $\ell_{ij},\ell_{ik},\ell_{jk}$, the terms
  $a_{ijk} \ell_{ij} \ell_{ik} \ell_{jk}$ are mutually independent
  and equal to $\pm1$ with the same probability.  This means that 
  \[ \sum_{i=1}^n \sum_{j=1}^n  \sum_{k=1}^n a_{ijk} \ell_{ij} \ell_{ik} \ell_{jk} \]
  has distribution $S_{n^3}$.
  Recall that $c = (6 \ln{2})^{1/2}$ and  set $\lambda = c n^{5/2}$.  Applying the deviation inequality yields
  \[ \Pr \left[ \sum_{i=1}^n \sum_{j=1}^n  \sum_{k=1}^n a_{ijk} \, \ell_{ij} \, \ell_{ik} \, \ell_{jk} > \lambda \right]
  < \exp(-\lambda^2/2n^3) = 2^{-3n^2}. \]
  There are $2^{3n^2}$ choices for the $3n^2$ lines, thus by the union bound we obtain
  \[ \Pr \left[ \sum_{i=1}^n \sum_{j=1}^n  \sum_{k=1}^n a_{ijk} \, \ell_{ij} \, \ell_{ik} \, \ell_{jk} > \lambda
    \text{ for some line }\right] < 2^{3n^2} \cdot 2^{-3n^2} =1. \]
  Consequently, there exists a matrix
  $[a_{ijk}]$ such that
  \[ \sum_{i=1}^n \sum_{j=1}^n  \sum_{k=1}^n a_{ijk} \, \ell_{ij} \, \ell_{ik} \, \ell_{jk} \leq \lambda = c n^{5/2} \]
  for all $3n^2$ switches of the lines $\ell_{ij}$, $\ell_{ik}$, $\ell_{jk}$. 

 \smallskip
 \emph{Lower bound.} The strategy can be viewed as adding up the discrepancies of $n$ square boards,
 over $n$ planes, $z=1,\ldots,n$, of the cubic lattice, and then using the lower bound in Theorem~\ref{thm:square}
 for each one. Note that we use only $2n^2$ out of the $3n^2$ switches when playing the game (\ie, $2n$ for each
 horizontal plane). 
\end{proof}

\subsection{Open questions} \label{sec:open}

Two interesting questions remain:  

\begin{enumerate}  \itemsep 2pt

\item Is it possible to close the gap between the asymptotic upper and lower bounds in the
  original version of the game? In particular, is it true that $F(n) \geq (1 + o(1)) \, n^{3/2}$?

\item What is the maximum discrepancy achievable in the game on the hyperbolic planar board
  considered in Theorem~\ref{thm:hyperbola}?
Is it $\sim n$, or $\sim n \, \sqrt{\log n}$ or in between?

\end{enumerate}

\end{document}